\documentclass[draft]{amsart}

\usepackage{hyperref}
\usepackage{mathtools}
\mathtoolsset{showonlyrefs}
\usepackage[foot]{amsaddr}
\usepackage{bm}

\theoremstyle{plain}
\newtheorem{theorem}{Theorem}
\newtheorem{lemma}{Lemma}

\theoremstyle{definition}
\newtheorem{definition}{Definition}

\theoremstyle{remark}
\newtheorem{remark}{Remark}

\newcommand\diff{d}
\newcommand{\nset}{\mathbb{N}}
\newcommand\rset{\mathbb{R}}
\DeclareMathOperator{\bary}{bar}
\DeclareMathOperator{\ave}{\mathbb{E}}

\begin{document}

\title[Fr{\'e}chet barycenters and a law of large numbers in~$\mathcal{P}(\rset)$]%
{Fr{\'e}chet Barycenters and a Law of Large Numbers\\
  for Measures on the Real Line$^*$}
\thanks{$^*$\,The research was carried out at the IITP RAS at the expense of the Russian Science Foundation (project № 14-50-00150).}
\author{Alexey Kroshnin$^{1,2}$}
\author{Andrei Sobolevski$^{1,3}$}
\address{$^1$\,Institute for Information Transmission Problems of RAS (Kharkevich Institute)}
\address{$^2$\,Moscow Institute of Physics and Technology}
\address{$^3$\,National Research University ``Higher School of Economics''}
\email{sobolevski@iitp.ru}
\date{December 2015}

\begin{abstract}
	Endow the space $\mathcal{P}(\rset)$ of probability measures on~$\rset$ with a transportation cost $J(\mu, \nu)$ generated by a translation-invariant convex cost function.
	For a probability distribution on~$\mathcal{P}(\rset)$ we formulate a notion of average with respect to this transportation cost, called here the \emph{Fr{\'e}chet barycenter}, prove a version of the law of large numbers for Fr{\'e}chet barycenters, and discuss the structure of~$\mathcal{P}(\rset)$ related to the transportation cost~$J$.
\end{abstract}

\maketitle

\section{Introduction}
\label{sec:introduction}

The law of large numbers is probably the oldest result in statistics:
already by Kepler's time, the sample arithmetic mean became universally accepted as estimator for a quantity whose measurements are subject to errors (see, e.g., \cite{Sheynin.O:2005}).  
About a century later, in Part~4 of his \textit{Ars Conjectandi}, published posthumously in 1713 \cite{Bernoulli.J:1713,Sheynin.O:2005}, Jacob Bernoulli gave a rigorous proof that for two given outcomes, the probability of either outcome can be determined by averaging the number of its occurrences over a large sample.

Note that Bernoulli's argument is based on embedding the two distinct outcomes $\{0, 1\}$ into the real line $\rset$ whose affine structure is then used to perform averaging.
To use a similar approach in non-affine spaces, such as collections of geometric shapes, Maurice Fr{\'e}chet introduced in his memoir \cite{Frechet.M:1948} a notion of averaging on a general metric space $(M, d)$: for a Borel measure~$\mu$ its \emph{Fr{\'e}chet mean} is the global minimum of $\int d^2(x, \cdot)\, \diff\mu(x)$.
This construction reduces to the conventional mean if the space $(M, d)$ is Euclidean and the measure $\mu$ has finite second moment, but in more complex situations Fr{\'e}chet means may fail to exist or to be unique.

In this paper we consider averaging in the space $\mathcal{P}(M)$ of measures over a metric space~$M$, using a transport optimization procedure to define a suitable concept of a ``typical element'' \cite[p.~224ff.]{Frechet.M:1948}, which extends the notion of Fr{\'e}chet mean.
For the first time a construction of this kind was introduced by M.~Agueh and G.~ Carlier in \cite{Agueh.M:2011}: a \emph{Wasserstein barycenter} of a family of measures on the Euclidean space~$\rset^d$ is defined as the Fr{\'e}chet mean using the $2$-Wasserstein distance $W_2$ on~$\mathcal{P}(\rset^d)$, which is given by minimization of the mean-square displacement (a precise definition is recalled on p.~\pageref{def:two_wasserstein}).
In \cite{Agueh.M:2011}, the authors establish existence, uniqueness, and regularity results for the Wasserstein barycenter and, when $d = 1$, provide an explicit formula for the Wasserstein barycenter in terms of quantile functions of the measures involved.

Here we limit ourselves to the case $d = 1$ but take a general transportation cost
\begin{equation}
	J(\mu, \nu) = \inf_{\substack{0 \le \gamma \in \mathcal{P}(\rset^2)\colon\\ \pi_\#^x\gamma = \mu, \pi_\#^y\gamma = \nu}} \int g(x - y) \, \diff\gamma(x, y) \ge 0,
\end{equation}
where~$g(\cdot) \ge 0$ is a strictly convex function that satisfies $g(0) = 0$.
Since $J(\mu, \nu) = 0$ iff $\mu = \nu$, this cost quanifiies separation between measures $\mu$ and~$\nu$ in $\mathcal{P}(\rset)$ but does not necessarily satisfy the triangle inequality; in \cite{Frechet.M:1948} such measures of separation (with an additional requirement of symmetry, $J(\mu, \nu) = J(\nu, \mu)$) are called ``\textit{\'ecarts}'' to distinguish them from \emph{distances}, for which the triangle inequality is satisfied.

Let a measure $\nu$ be fixed and $\bm\mu$ be a random element of~$\mathcal{P}(\rset)$ with distribution $P_\mu$.
We introduce a notion of Fr{\'e}chet typical element of~$P_\mu$ with respect to~$J(\cdot, \cdot)$, which we propose to call the \emph{Fr{\'e}chet barycenter} of~$P_\mu$.
It is defined as any measure~$\nu$ for which the expected cost
\begin{equation}
	\label{eq:34}
	\ave J(\bm\mu, \nu) = \int_{\mathcal{P}(\rset)} J(\mu, \nu)\, \diff P_\mu
\end{equation}
attains its minimum over~$\mathcal{P}(\rset)$.
Rigorous definitions of such a distribution and an integral are formulated in Section~\ref{sec:baryc-family-meas}.

Suppose $\ave J(\bm\mu, \cdot)$ is not identically equal to~$+\infty$ on~$\mathcal{P}(\rset)$. 
Under this assumption the strict convexity of~$g$ ensures that a Fr{\'e}chet barycenter~$\nu^*$ is unique, and 
the one-dimensional setting allows to provide it with an explicit expression.
Namely, the \emph{quantile function} of the distribution~$\nu^*$ is given by
\begin{equation}
	\psi\colon (0, 1) \ni x \mapsto \arg\min_{y\in\rset} \ave g(F^{-1}_{\bm\mu}(x) - y),
\end{equation}
where $F^{-1}_{\bm\mu}$ is the quantile function (i.e., the inverse of the cumulative distribution function) of the random measure~$\bm\mu$.
In particular, when $g(x, y) = (x - y)^2$ the function~$\psi$ reduces to the usual arithmetic average of quantile functions, which was shown in \cite{Agueh.M:2011} to give the Wasserstein barycenter on~$\rset$.

In other words, \emph{every quantile of the Fr{\'e}chet barycenter~$\nu^*$, defined with respect to a transportation cost on~$\mathcal{P}(\rset)$ generated by the cost function~$g(\cdot)$, is given by the typical value of the corresponding quantile, defined on~$\rset$ with respect to the function~$g(\cdot)$}.
We establish this result first for Fr{\'e}chet barycenters of finite samples (Theorem~\ref{thm:emp_bar}) and then for an arbitrary probability distribution~$P_\mu$ on~$\mathcal{P}(\rset)$ such that $\ave J(\bm\mu, \nu) < +\infty$ for at least one measure $\nu$ (Theorem~\ref{thm:prob_bar}).

These results are then used to obtain the following form of a law of large numbers: \emph{the Fr{\'e}chet barycenter of an independent sample of size~$n$ from the distribution~$P_\mu$ weakly converges as~$n\to\infty$ to the Fr{\'e}chet barycenter of~$P_\mu$ itself}.

To see this we first establish a similar statement for Fr{\'e}chet typical elements in~$\rset$ with respect to the function~$g$.
Let the distribution $P_X$ on~$\rset$ be such that the function $\xi(\cdot)= \ave g(\bm X - \cdot)$ is finite for all~$x\in \rset$.
Then the Fr{\'e}chet typical element of the i.i.d.\ sample $\bm X_1, \bm X_2, \dots, \bm X_n$ from the distribution~$P_X$ converges to~$x^* := \arg \min_{x\in\rset} \xi(x)$ as $n \to \infty$  (Theorem~\ref{thm:3}).
This implies the law of large numbers for random measures on~$\rset$, which is again proved ``quantile-wise'' (Theorem~\ref{thm:4}).

This law of large numbers is then given another formulation: \emph{the sequence of ``empirical'' Fr{\'e}chet barycenters $\bm\nu_n$ converges to~$\nu^*$ in the sense that $J(\bm\nu_n, \nu^*) \to 0$ almost surely} (Theorem~\ref{thm:strong_conv}).
This notion of convergence is shown to be somewhat stronger than the weak convergence of measures (Theorem~\ref{thm:strongweak}).

The results reported here should be compared with recent results of T.~Le Gouic and J.-M.~Loubes \cite{LeGouic.T:2015} (see also the earlier series of works \cite{Boissard.E:2015,LeGouic.T:2013,Bigot.J:2015}, where the law of large numbers based on Wasserstein barycenters in considered in the context of deformation models).
The paper \cite{LeGouic.T:2015} contains a general consistency result for Wasserstein barycenters when a sequence of probability distributions over the $2$-Wasserstein space $(\mathcal{P}(M), W_2)$ converges with respect to the ``derivate'' $2$-Wasserstein distance (a $2$-Wasserstein distence defined over $(\mathcal{P}(M), W_2)$ taken as a metric space itself).
It should be stressed that in~\cite{LeGouic.T:2015} the base space $M$ is assumed to be an arbitrary locally compact geodesic space.

Our results are comparable to the construction of~\cite{LeGouic.T:2015} in the special situation when the two settings agree, i.e., when $M = \rset$ and $g(x - y) = (x - y)^2$.
Note that we do not consider general convergence of measures, but only convergence of an empirical measure of an i.i.d.\ sample to the distribution from which it is drawn.
In this setting conditions of our Theorems~\ref{thm:prob_bar} and~\ref{thm:4} are somewhat easier to check then the condition of convergence under the ``derivate'' $2$-Wasserstein metric. 
On the other hand, we use the one-dimensional geometry of the base space in an essential way, which is not readily extendable to a general manifold.

The paper is organized as follows.
In Section~\ref{sec:basic-defin-notat} we introduce some standard definitions and notation and recall that when the cost function is defined on the real line and convex, the optimal transport plan is given by a monotone map defined solely by the marginal measures irrespective of the specific form of the cost function.
Then we define in Section~\ref{sec:gener-baryc-conv} the generalized barycenter of a finite set of measures in $\mathcal{P}(\rset)$ and provide it with an explicit representation.
In Section~\ref{sec:baryc-family-meas} this construction is extended to barycenters of continuous families of measures in~$\mathcal{P}(\rset)$.
Then the central result of this paper is proved in Section~\ref{sec:weak_convergence}: the convergence of empirical barycenters for an i.i.d.\ sequence of random measures to the ``barycenter expected value'' of the corresponding distribution.
In order to prove it, we establish first a version of the law of large numbers for a suitable ``nonlinear averaging'' (Theorem~\ref{thm:3}) which has some independent interest.
Section~\ref{sec:strong_convergence} establishes convergence with respect to the transportation cost~$J(\cdot, \cdot)$, a property which turns out to be more informative than just the weak convergence.
An Appendix contains proof of a technical measure-theoretic result used in Section~\ref{sec:baryc-family-meas}.

\section{Basic facts on mass transportation on the real line}
\label{sec:basic-defin-notat}

For a measurable space $X$ denote the space of probability measures on~$X$ by $\mathcal{P}(X)$.  
In particular, if the space $X$ is topological, we assume it to be endowed with the standard Borel $\sigma$-algebra $\mathcal{B}(X)$.  
For two measurable spaces $X$, $Y$, a measurable map $T\colon X \to Y$ induces a map $T_\#\colon \mathcal{P}(X)\to \mathcal{P}(Y)$ given by $T_\#\mu(A) := \mu(T^{-1}(A))$ for any measurable $A\subset Y$.
Recall that for any integrable function $f$
\begin{equation}
	\label{eq:1}
	\int_Y f(y)\, \diff (T_\#\mu) = \int_X f(T(x))\, \diff\mu.
\end{equation}

For two measures $\mu, \nu \in \mathcal{P}(X)$ define the set of \emph{transport plans} taking $\mu$ to~$\nu$ as
\begin{equation}
	\label{eq:2}
	\Pi(\mu, \nu) := \Bigl\{\gamma \in \mathcal{P}(X\times X)\colon \pi^x_\#\gamma = \mu,\   \pi^y_\#\gamma = \nu\Bigr\},
\end{equation}
where $\pi^x$ and~$\pi^y$ are projections of~$X\times X$ to the first and second factor respectively.  
Observe that $\Pi(\mu, \nu)$ is always nonempty because it contains the direct product measure~$\mu\times\nu$. 

Fix a measurable function $c\colon X\times X \to \rset$ and call it the \emph{cost function}.
The \emph{transportation cost} of a transport plan~$\gamma$ is defined as
\begin{equation}
	\label{eq:3}
	K(\gamma) = \int_{\rset\times\rset} c(x, y)\, \diff\gamma.
\end{equation}
The \emph{Monge--Kantorovich problem} for given $\mu, \nu \in \mathcal{P}(X)$ consists in minimizing the \emph{transportation cost} $K(\gamma)$ over all $\gamma \in \Pi(\mu, \nu)$.  
A transport plan $\gamma^*$ is called \emph{optimal} if $K$ attains its minimum over $\Pi(\mu, \nu)$ at~$\gamma^*$.
If moreover $\gamma^* = (\mathrm{id}\times T^*)_\#\mu$ for some measurable~$T^*$, then $T^*$ is called the \emph{optimal transport map}.
Observe that $T^*_\#\mu = \nu$ if $T^*$ exists.

Recall that a measure $\mu$ on~$\rset$ is characterized by its (left-continuous) \emph{cumulative distribution function} $F_\mu(x) := \mu((-\infty, x))$, and its (left-continuous) inverse $F^{-1}_\mu(y) := \inf\{x\colon F(x) \ge y\}$ for $0 < y < 1$ is called the \emph{quantile function} of the measure~$\mu$.
In this paper we consider $X = \rset$ and take $c(x, y) = g(x - y)$, where the function $g$ is strictly convex.
We assume that $g$ attains its minimum at $x = 0$ (unique because $g$ is strictly convex) and, without loss or generality, that $g(0) = 0$.
In this setting the Monge--Kantorovich problem has a well-known explicit solution, which we recall here.

\addtocounter{theorem}{-1}
\begin{theorem}
	\label{thm:1}
	Let the infimum of transportation cost $K(\gamma)$ on $\Pi(\mu, \nu)$ be finite.
	Then an optimal transport plan exists and has a uniquely defined form 
	\begin{equation}
		\label{eq:4}
		\gamma^* = (F^{-1}_\mu \times F^{-1}_\nu)_\#\left.\mathcal{L}\right|_{(0, 1)},
	\end{equation}
	where $\left.\mathcal{L}\right|_{(0, 1)}$ is the standard Lebesque measure on~$(0, 1)$.  
\end{theorem}

We first prove the following lemma, which is used repeatedly in the sequel.
\begin{lemma}
	\label{lm:convex}
	For a convex function $g(\cdot)$, inequalities $x < x + \delta < y$ and $x < y - \delta < y$ imply that
	\begin{equation}
		\label{eq:32}
		g(x) + g(y) \ge g(x + \delta) + g(y - \delta),
	\end{equation}
	and when the function~$g$ is strictly convex, this inequality is strict.
\end{lemma}

\begin{proof}
	Indeed, if $\lambda = \delta/(y - x)$, then clearly $0 < \lambda < 1$ and
	\begin{displaymath}
		\begin{aligned}
			x + \delta &= (1 - \lambda) x + \lambda y, \\
			y - \delta &= \lambda x + (1 - \lambda) y.
		\end{aligned}
	\end{displaymath}
	Therefore convexity can be used to get
	\begin{displaymath}
		\begin{aligned}
			g(x + \delta) &\le (1 - \lambda) g(x) + \lambda g(y), \\
			g(y - \delta) &\le \lambda g(x) + (1 - \lambda) g(y),
		\end{aligned}
	\end{displaymath}
	which implies the statement.
	The case of strict convexity is treated similarly.
\end{proof}

\begin{proof}[Sketch of proof of Theorem~\protect\ref{thm:1}]
	Consider points $x_1$, $x_2$, $y_1$, and~$y_2$ in~$\rset$ such that $x_1 < x_2$ and $y_1 < y_2$.
	Then both $x_1 - y_1$ and~$x_2 - y_2$ lie strictly between $x_1 - y_2$ and $x_2 - y_1$, which in view of the strict convexity of~$g$ and of Lemma~\ref{lm:convex} implies that
	\begin{equation}
		\label{eq:5}
		g(x_1 - y_1) + g(x_2 - y_2) < g(x_1 - y_2) + g(x_2 - y_1).
	\end{equation}
	As $\inf K(\gamma)$ over~$\Pi(\mu, \nu)$ is finite, there cannot exist sets $A$, $B$ with $\gamma^*(A) > 0$, $\gamma^*(B) > 0$ such that $(x_1 - x_2) (y_1 - y_2) < 0$ for all $(x_1, y_1) \in A$ and $(x_2, y_2) \in B$.
	Thus, if an optimal map~$T(\cdot)$ exists, it must be nondecreasing $\mu$-a.e., and the support of~$\gamma^*$ must have the form $\{(x(t), y(t))\colon t\in (a, b)\}$ for some interval~$(a, b)$. 
	As $\gamma^*\in \Pi(\mu, \nu)$, such a transport plan is necessarily given by~\eqref{eq:4}.
\end{proof}

\begin{remark}
	Observe that if the optimal map $T(\cdot)$ exists (in particular, if $\mu$ does not have atoms), then $\mu$-a.e.
	\begin{equation}
		\label{eq:6}
		T(x) = F^{-1}_\nu(F_\mu(x)).
	\end{equation}
	Hence the optimal map does not depend on the specific form of the cost function~$g$ (apart from the fact that it is strictly convex).
\end{remark}

\begin{remark}
	If $g(\cdot)$ is convex but not strictly so, then formulas \eqref{eq:4} and~\eqref{eq:6} still give an optimal transport plan and an optimal map (provided the latter is well defined), but there may exist other optimal transport plans and maps.
\end{remark}

\section{Fr{\'e}chet barycenters for convex cost functions on~$\rset$}
\label{sec:gener-baryc-conv}

Define a functional on $\mathcal{P}(\rset)\times \mathcal{P}(\rset)$ by
\begin{equation}
	\label{eq:8}
	J(\mu, \nu) := \inf \{K(\gamma)\colon \gamma\in \Pi(\mu, \nu)\} = \int_0^1 g(F^{-1}_\mu(x) - F^{-1}_\nu(x))\, \diff x
\end{equation}
(where we have used the above theorem) or, if an optimal map~$T$ exists, by
\begin{equation}
	\label{eq:9}
	J(\mu, \nu) = \int_\rset g(x - T(x))\, \diff\mu.
\end{equation}

One can show that if $c(\cdot, \cdot)$ is a distance function on~$\rset$, then $J$ satisfies the triangle inequality on~$\mathcal{P}(\rset)$, i.e., it is itself a distance.
Another important particular case\label{def:two_wasserstein} is when $c(x, y) = (x - y)^2$: in this case $\sqrt{J(\mu, \nu)}$ gives a distance on~$\mathcal{P}(\rset)$, called the \emph{$2$-Wasserstein distance} and denoted $W_2(\mu, \nu)$.

\begin{definition}
	\label{def:barycenter}
	Consider a finite set $\mu_1, \mu_2, \dots, \mu_n$ of measures in~$\mathcal{P}(\rset)$ and the \emph{weights} $\lambda_1 > 0, \lambda_2 > 0, \dots, \lambda_n > 0$.
	The \emph{Fr{\'e}chet barycenter} $\bary (\mu_i, \lambda_i)_{1\le i\le n} \in \mathcal{P}(\rset)$ with respect to the cost function~$c(x, y) = g(x - y)$ is a measure that minimizes
	\begin{equation}
		\label{eq:10}
		\sum_{1\le i\le n} \lambda_i J(\mu_i, \nu)
	\end{equation}
	over $\nu \in \mathcal{P}(\rset)$.
\end{definition}

Without loss of generality we can assume the weights to be normalized so that $\sum_i \lambda_i = 1$.
In particular, the definition of Wasserstein barycenter given in \cite{Agueh.M:2011} is recovered when $J(\mu, \nu) = W_2^2(\mu, \nu)$.

\begin{theorem}
	\label{thm:emp_bar}
	Suppose $\sum_i \lambda_i J(\mu_i, \nu) < +\infty$ for some $\nu \in \mathcal{P(\rset)}$.
	Then the Fr{\'e}chet barycenter exists and is uniquely defined as $\nu^* = \psi_\# \left.\mathcal{L}\right|_{(0, 1)}$, where for $0 < x < 1$
	\begin{equation}
		\label{eq:11}
		\psi(x) := \arg\min_{y\in\rset} \sum_{1\le i\le n} \lambda_i g(F^{-1}_{\mu_i}(x) - y).
	\end{equation}
\end{theorem}

\begin{proof}
	Observe first that $J \ge 0$ and therefore the condition of the theorem ensures that $\inf \sum_i \lambda_i J(\mu_i, \nu)$ over $\nu\in \mathcal{P}(\rset)$ is finite.

	Now fix some $n$-tuple $a = (a_1, a_2, \dots, a_n) \in \rset^n$ and consider the function
	\begin{equation}
		\label{eq:12}
		f_a(y) := \sum_{1\le i\le n} \lambda_i g(a_i - y).
	\end{equation}
	The assumptions on~$g$ imply that for all~$a$ this function is also strictly convex and attains a unique minimum on~$\rset$, so the function $\psi$ introduced in~\eqref{eq:11} is well-defined.

	Let us show that $\psi$ is nondecreasing.
	Indeed, assume on the contrary that  $\psi(x_1) > \psi(x_2)$ for some $0 < x_1 < x_2 < 1$.
	Then convexity of~$g$, monotonicity of~$F^{-1}_{\mu_i}(\cdot)$, and Lemma~\ref{lm:convex} imply that
	\begin{multline}
		\label{eq:13}
		g(F^{-1}_{\mu_i}(x_1) - \psi(x_2)) + g(F^{-1}_{\mu_i}(x_2) - \psi(x_1)) \\
		\le g(F^{-1}_{\mu_i}(x_1) - \psi(x_1)) + g(F^{-1}_{\mu_i}(x_2) - \psi(x_2)),
	\end{multline}
	for all~$1\le i\le n$, where $y_1 = \psi(x_2)$ and $y_2 = \psi(x_1)$ (the inequality will be nonstrict if $F^{-1}_{\mu_i}(x_1) = F^{-1}_{\mu_i}(x_2)$).
	Multiplying each of these inequalities by~$\lambda_i$ and summing over~$i$, we get
	\begin{multline}
		\label{eq:14}
		\sum_{1\le i\le n} \lambda_i \bigl[g(F^{-1}_{\mu_i}(x_1) - \psi(x_2))
		+ g(F^{-1}_{\mu_i}(x_2) - \psi(x_1))\bigr] \\
		\le \sum_{1\le i\le n} \lambda_i \bigl[g(F^{-1}_{\mu_i}(x_1) - \psi(x_1))
		+ g(F^{-1}_{\mu_i}(x_2) - \psi(x_2))\bigr].
	\end{multline}
	On the other hand, $\psi(x_1)$ and $\psi(x_2)$ are the unique minima of the right-hand side of~\eqref{eq:11} for the respective values $x_1$ and~$x_2$.
	Thus
	\begin{equation}
		\label{eq:15}
		\sum_{1\le i\le n} \lambda_i g(F^{-1}_{\mu_i}(x_1) - \psi(x_2))
		> \sum_{1\le i\le n} \lambda_i g(F^{-1}_{\mu_i}(x_1) - \psi(x_1))
	\end{equation}
	and
	\begin{equation}
		\label{eq:16}
		\sum_{1\le i\le n} \lambda_i g(F^{-1}_{\mu_i}(x_2) - \psi(x_1))
		> \sum_{1\le i\le n} \lambda_i g(F^{-1}_{\mu_i}(x_2) - \psi(x_2)).
	\end{equation}
	Adding the latter two inequalitites term by term, we obtain a contradiction with~\eqref{eq:14}, which proves that $\psi$ is nondecreasing.

	Define now $F(x) := \inf\{y\in (0, 1)\colon \psi(y) \ge x\}$ whenever the set in the r.h.s.\ is non-empty, and $F(x) = 1$ if $\psi(y) < x$ for all $0 < y < 1$.
	Then $F(-\infty) = 0$, $F(+\infty) = 1$ and $F$ is left-continuous, i.e., there exists a measure $\nu^* \in \mathcal{P}(\rset)$ such that $\nu^*((-\infty, x)) = F(x)$ for all~$x\in\rset$.
	Note that $\psi(\cdot)$ as defined in~\eqref{eq:11} is left-continuous because all $F^{-1}_{\mu_i}(\cdot)$ are, so $F^{-1}(x) = \psi(x)$ on~$(0, 1)$.

	We now check that $\nu^*$ is a Fr{\'e}chet barycenter.
	For any $\nu \in \mathcal{P}(\rset)$,
	\begin{multline}
		\label{eq:17}
		\sum_{1\le i\le n} \lambda_i J(\mu_i, \nu)
		= \sum_{1\le i\le n} \lambda_i\int_0^1 g(F^{-1}_{\mu_i} - F^{-1}_\nu)\, \diff x \\
		= \int_0^1 \sum_{1\le i\le n} \lambda_i g(F^{-1}_{\mu_i} - F^{-1}_\nu)\, \diff x
		\ge \int_0^1 \sum_{1\le i\le n} \lambda_i g(F^{-1}_{\mu_i} - \psi)\, \diff x \\
		= \sum_{1\le i\le n} \lambda_i \int_0^1 g(F^{-1}_{\mu_i} - F^{-1}_{\nu^*})\, \diff x
		= \sum_{1\le i\le n}\lambda_i J(\mu_i, \nu^*).
	\end{multline}
	The latter quantity is finite, and the strict convexity of~$g$ ensures that the inequality of the middle line is strict unless $\nu = \nu^*$.
	Thus $\nu^* = \bary(\mu_i, \lambda_i)_{1\le i\le n}$.
\end{proof}

\begin{remark}
	If $g(\cdot)$ is smooth, then the function~$\psi$ is determined from the equation
	\begin{equation}
		\label{eq:18}
		\sum_i \lambda_i g'(F^{-1}_{\mu_i} - \psi) = 0.
	\end{equation}
	In particular when $g(x) = x^2$ we recover the formula $\psi(x) = \sum_i \lambda_i F^{-1}_{\mu_i}(x) $ for the Wasserstein barycenter \cite{Agueh.M:2011}.
\end{remark}

\begin{remark}
	If $g$ is convex but not strictly so, then $\arg\min$ in~\eqref{eq:11} may be attained on an interval rather than at a single point.
	Define $\psi(x)$ to be the left endpoint of this interval; we will then still obtain a Fr{\'e}chet barycenter as the measure $\nu^*$ for which $\nu^*((-\infty, x)) = \psi^{-1}(x)$, but this barycenter will not necessarily be unique.
\end{remark}

Now we show that the Fr{\'e}chet barycenter of fixed number of arguments~$m$ is weakly continuous with respect to weights $\lambda_1, \dots, \lambda_m$ and weak convergence of measures.

\begin{lemma}
	\label{lem:weak_conv}
	Take a sequence of measures $\{\nu_n\}_{n\in \nset}$ and a measure $\nu^*$ such that $F^{-1}_{\nu_n}(y)\to F^{-1}_{\nu^*}(y)$ for all $y\in M$, where $M$ is a dense subset of~$(0, 1)$. Then $\nu_n\rightharpoonup \nu^*$.
\end{lemma}

\begin{proof}
	We show that the condition of lemma implies that $F_{\bm\nu_n}\to F_{\nu^*}$ at all points where $F_{\nu^*}$ is continuous what is equivalent to the weak convergence of measures.
	Indeed, let this convergence fail at some $x_0$ where $F_{\nu^*}$ is continuous.
	Without loss of generality we can assume that there exists a subsequence $n_k$ such that $F_{\bm\nu_{n_k}}(x_0) > F_{\nu^*}(x_0) + \epsilon$ for a suitable $\epsilon > 0$.
	Take $y\in M$ such that $F_{\nu^*}(x_0) < y < F_{\nu^*}(x_0) + \epsilon$.
	Then continuity of $F_{\nu^*}$ at~$x_0$ implies that there exists $\gamma > 0$ such that $F_{\nu^*}(x) < y$ whenever $x < x_0 + \gamma$.
	Then $F^{-1}_{\nu^*}(y) \ge x_0 + \gamma$ whereas $F^{-1}_{\nu_{n_k}}(y) \le x_0$, which in the limit $n_k \to \infty$ gives a contradiction.
\end{proof}

\begin{lemma}
	\label{lem:scalar_cont}
	Taking sequences $x_i^n \to x_i$ and weights $\lambda_i^n\to \lambda_i$, $1\le i\le m$, define points
	\begin{equation}
		\overline x^n := \arg\min_{y\in \rset} \sum_{1\le i\le m} \lambda_i^n g(x_i^n - y),\quad n = 1, 2, \dots
	\end{equation}
	and $x^* := \arg\min_{y\in \rset} \sum_{1\le i\le m} \lambda_i g(x_i - y)$. Then $\overline x^n\to x^*$.
\end{lemma}

\begin{proof}
	Consider arbitrary $y > x^*$. The continuity of~$g$ ensures that
	\begin{multline}
		\lim \sum_{1\le i\le m} \lambda_i^n g(x_i^n - x^*) = \sum_{1\le i\le m} \lambda_i g(x_i - x^*) < \\
		\sum_{1\le i\le m} \lambda_i g(x_i - y) = \lim \sum_{1\le i\le n} \lambda_i^n g(x_i^n - y),
	\end{multline}
	hence beginning from some $n$ it holds 
	\begin{equation}
		\sum_{1\le i\le m} \lambda_i^n g(x_i^n - x^*) < \sum_{1\le i\le n} \lambda_i^n g(x_i^n - y).
	\end{equation}
	Due to the convexity (and, consequently, unimodality) of~$\sum_{1\le i\le n} \lambda_i^n g(x_i^n - \cdot)$, this implies that $\overline x^n \le y$ what means $\limsup \overline x^n \le y$. The remaining case $y < x^*$ can be treated similarly and this completes the proof.
\end{proof}

\begin{theorem}
	\label{thm:continuity}
	Let weights $\lambda_i^n\to \lambda_i$ and measures $\mu_i^n\rightharpoonup \mu_i$, $1\le i\le m$; then 
	\begin{equation}
		\bary(\mu_i^n, \lambda_i^n)_{1\le i\le m}\rightharpoonup \bary(\mu_i, \lambda_i)_{1\le i\le m}.
	\end{equation}
\end{theorem}

\begin{proof}
	The weak convergence of measures $\mu_i^n$ means that for all $1\le i\le m$ it holds $F_{\mu_i^n}^{-1}(x)\to F_{\mu_i}^{-1}(x)$ if $F_{\mu_i}^{-1}$ is continuous at point~$x$, i.e., for a.e. $x\in (0, 1)$. Now combining Lemmas~\ref{lem:weak_conv} and~\ref{lem:scalar_cont} one can obtain the statement.
\end{proof}

\section{The Fr{\'e}chet barycenter of a continuous distribution}
\label{sec:baryc-family-meas}

Now we extend the notion of Fr{\'e}chet barycenter to continuous families of measures, which allows to define an ``expected value'' for a probability distribution over~$\mathcal{P}(\rset)$.
Endow $\mathcal{P}(\rset)$ with topology of weak convergence, and let $\mathcal{B}(\mathcal{P}(\rset))$ be the corresponding Borel $\sigma$-algebra.
We need a technical measurability lemma whose proof is postponed to Appendix.

\begin{lemma}
	\label{lm:1}
	The function $K(\mu, x) := F^{-1}_\mu(x)$, where $\mu \in \mathcal{P}(\rset)$ and $0 < x < 1$, is measurable with respect to the product $\sigma$-algebra $\mathcal{B}(\mathcal{P}(\rset)) \otimes \mathcal{B}((0, 1))$.
\end{lemma}

Let now $\bm\mu$ be a random element of~$\mathcal{P}(\rset)$ distributed according to a law~$P_\mu$.
Recall from the last section that $g(\cdot)$ is assumed to be a strictly convex function attaining on~$\rset$ a minimal value $g(0) = 0$.
For the functional $J(\mu, \nu)$ defined in~\eqref{eq:8}
\begin{equation}
	\label{eq:24}
	\ave J(\bm\mu, \nu) = \int_{\mathcal{P}(\rset)} J(\mu, \nu)\, \diff P_\mu \in [0, +\infty].
\end{equation}

\begin{definition}
	Consider the problem of minimizing $\ave  J(\bm\mu, \nu)$ over~$\mathcal{P}(\rset)$ and denote its solution by $\bary(P_\mu)$ or~$\nu^*$ for short.  
	We call the measure $\bary(P_\mu)$ the \emph{Fr{\'e}chet barycenter} of the distribution~$P_\mu$.
\end{definition}

\begin{lemma}
	\label{lm:2}
	Let $\bm X$ be a random element of~$\rset$ with distribution~$P_X$.
	For any $x\in \rset$ consider the function
	\begin{equation}
		\label{eq:30}
		\xi(x) := \ave g(\bm X - x)
	\end{equation}
	taking values in $[0, +\infty]$.
	This function attains a unique minimum provided $\xi(x) < +\infty$ for some~$x$.
\end{lemma}

\begin{proof}
	Let's show that $\xi(\cdot)$ is lower semi-continuous. Indeed consider any point $x\in \rset$ and sequence $x_n\to x$. Then using continuity of~$g(\cdot)$ and Fatou's lemma one can obtain
	\begin{equation}
		\xi(x) := \ave g(\bm X - x) = \ave \lim g(\bm X - x_n)\le \liminf \ave g(\bm X - x_n) := \liminf \xi(x_n)
	\end{equation}
	which is the definition of lower semi-continuity.
	Clearly $\xi(\cdot)$ is strictly convex because so is~$g(\cdot)$. Note also that $\lim_{x\to\pm\infty} g(x) = +\infty$ implies that $\lim_{x\to\pm\infty} \xi(x) = +\infty$. Thus a minimum of~$\xi(\cdot)$ exists and is unique, due to the strict convexity of~$\xi(\cdot)$.
\end{proof}

\begin{theorem}
	\label{thm:prob_bar}
	Suppose $\bm\mu \in \mathcal{P}(\rset)$ is a measure-valued random variable with distribution~$P_\mu$ and $\ave J(\bm\mu, \nu) < +\infty$ for some $\nu \in \mathcal{P}(\rset)$.
	Then there exists a unique solution $\nu^* = \psi_\#\left.\mathcal{L}\right|_{(0, 1)}$, where for a.e.\ $x\in (0, 1)$ \textup{(}\emph{cf.} \eqref{eq:11}\textup{)}
	\begin{equation}
		\label{eq:28}
		\psi(x) := \arg\min_{y\in\rset} \ave g(F^{-1}_{\bm\mu}(x) - y).
	\end{equation}
\end{theorem}

\begin{proof}
	Examination of equation~\eqref{eq:8} shows that for Lebesgue-a.e.\ $0 < x < 1$ there exists a value $y_x$ such that $\ave g(F^{-1}_{\bm\mu}(x) - y_x) < +\infty$.
	Therefore the function
	\begin{equation}
		\label{eq:29}
		\psi(x) := \arg\min_{y\in\rset} \ave g(F^{-1}_{\bm\mu}(x) - y)
	\end{equation}
	is defined according to Lemma~\ref{lm:1} for Lebesgue-a.e.\ $0 < x < 1$.
	Using the same argument as in Theorem~\ref{thm:emp_bar} (though with integrals over $P_\mu$ instead of sums over $\lambda_i$), one can show that the function~$\psi$ is nondecreasing.
	Thus it can be extended over the whole interval $(0, 1)$, e.g., by left continuity.
	Then $\psi$ generates a measure $\nu^* = \psi_\#\left.\mathcal{L}\right|_{(0, 1)} \in\mathcal{P}(\rset)$, where $F^{-1}_{\nu^*}(\cdot) = \psi(\cdot)$ a.e.

	We are left with a task of checking that $\nu^*$ is a Fr{\'e}chet barycenter.
	This can again be done in the same way as in Theorem~\ref{thm:emp_bar}, exchanging the order of integration over $0 < x < 1$ and integration over $\mathcal{P}(\rset)$ (which replaces summation over~$i$) by Fubini's theorem.
	In the process we employ the measurability lemma formulated at the beginning of this section.
	Uniqueness of the Fr{\'e}chet barycenter follows because $\psi$ (more precisely, its left-continuous version) is defined uniquely.
\end{proof}

\section{Weak convergence of the empirical Fr{\'e}chet barycenter}
\label{sec:weak_convergence}

First we prove a version of law of large numbers for a nonlinear version of averaging performed in terms of the function~$g$ over ordinary scalar random variables.
We then employ this result to prove convergence of quantiles for a sequence of empirical Fr{\'e}chet barycenters, which implies the weak convergence of the barycenters themselves.

\begin{theorem}
	\label{thm:3}
	Let $\{\bm X_n\}_{n\in\nset}$ be a sequence of i.i.d.\ real random variables with distribution~$P_X$ such that the function $x\mapsto \xi(x) = \ave g(\bm X - x)$ defined in~\eqref{eq:30} is finite for some $x\in\rset$.
	Let $x^* := \arg\min_{x\in\rset} \xi(x)$ and define the random variables
	\begin{equation}
		\label{eq:33}
		\overline{\bm X}_n := \arg\min_{x\in\rset} \sum_{1\le i\le n} g(\bm X_i - x),\qquad n = 1, 2, \dots
	\end{equation}
	Then $\overline{\bm X}_n$ converges to~$x^*$ a.s.
\end{theorem}

\begin{proof}
	It follows from Lemma~\ref{lm:2} that $x^*$ is well defined. Consider arbitrary $x > x^*$. By the strong law of large numbers one can obtain that a.s.
	\begin{equation}
		\lim \frac1n \sum_{1\le i\le n} g(\bm X_i - x^*) = \ave g(\bm X - x^*) := \xi(x^*) < \xi(x) = \lim \frac1n \sum_{1\le i\le n} g(\bm X_i - x).
	\end{equation}
	Now one can complete the proof in the same way as in Lemma~\ref{lem:scalar_cont}.
\end{proof}

\begin{theorem}
	\label{thm:4}
	Take a sequence $\{\bm\mu_n\}$ of random measures in~$\mathcal{P}(\rset)$, independent and identically distributed with the law~$P_\mu$ such that conditions of Theorem~\ref{thm:prob_bar} are satisfied.
	Then empirical Fr{\'e}chet barycenters $\bm\nu_n = \bary(\bm\mu_i, 1/n)_{1\le i\le n}$ weakly converge to $\nu^* := \bary(P_\mu)$ almost surely.
\end{theorem}

\begin{proof}
	Define a function $\psi(\cdot)$ as in Theorem~\ref{thm:prob_bar}:
	\begin{equation}
		\psi(x) := \arg\min_{y\in\rset} \ave g(F^{-1}_{\bm\mu}(x) - y)
	\end{equation}
	for a.e.\ $0 < x < 1$.
	Theorem~\ref{thm:3} ensures that the quantile $F^{-1}_{\bm\nu_n}(x)$ converges to~$\psi(x)$ almost surely for a.e.\ $0 < x < 1$.
	As $\psi(\cdot) = F_{\nu^*}^{-1}(\cdot)$ a.e., there exists a dense set $M \subset (0, 1)$ such that a.s. $F^{-1}_{\bm\nu_n}(x) \to F^{-1}_{\nu^*}(x)$ for all $x \in M$. By Lemma~\ref{lem:weak_conv} this implies the statement.
\end{proof}

\section{Convergence of the empirical Fr{\'e}chet barycenter\\ with respect to the transportation cost~$J$}
\label{sec:strong_convergence}

In this section we will additionally assume that the cost function $g(\cdot)$ satisfies the following condition: there exist nonnegative constants $A$, $B$ such that 
\begin{equation}
	\label{eq:growth_rate}
	g(x - y)\le A + B \bigl(g(x) + g(y)\bigr)
\end{equation}
for all~$x, y\in\rset$ (note that a similar condition with $A = 0$ appears in Fr{\'e}chet's memoir \cite[p.~228]{Frechet.M:1948}).
This condition is not exceedingly restrictive.
In particular, it is satisfied for cost functions of algebraic growth:

\begin{lemma}
	For $p > 0$ suppose that constants $0 < C_1 < C_2$, $x_0 > 0$ are such that $C_1 |x|^p < g(x) < C_2 |x|^p$ for any $|x| > x_0$.
	Then \eqref{eq:growth_rate} holds.
\end{lemma}

\begin{proof}
	Since $g(\cdot)$ is continuous, there exists a finite positive $C_0 := \max_{|x|\le x_0} g(x)$.
	Then for any $x$, $y$ on $\rset$
	\begin{align}
	g(x - y) &\le \max\{C_0, C_2 |x - y|^p\}\le C_0 + C_2 |x - y|^p\\
	|x - y|^p &\le 2^p \max\{|x|^p, |y|^p\} \le 2^p (|x|^p + |y|^p)\\
	|x|^p &\le \max\{x_0^p, \frac{g(x)}{C_1}\}\le x_0^p + \frac{g(x)}{C_1}
	\end{align}
	Gathering these results, we get the desired inequality:
	\[g(x - y) \le (C_0 + 2^{p+1} C_2 x_0^p) + 2^p \frac{C_2}{C_1} \bigl(g(x) + g(y)\bigr). \qedhere\]
\end{proof}

\begin{lemma}
	Let the relation $\mu \sim \nu$ be defined as $J(\mu, \nu) < +\infty$; then it is an equivalence on~$\mathcal{P}(\rset)$.
\end{lemma}

\begin{proof}
	The relation $\sim$ is obviously reflective.
	It is symmetric because $g(-x) \le A + Bg(x)$ according to~\eqref{eq:growth_rate} and therefore
	\begin{displaymath}
		J(\nu, \mu) = \int_0^1 g(F_\nu^{-1} - F_\mu^{-1})\,\diff x \le A + BJ(\mu, \nu) < +\infty.
	\end{displaymath}
	Finally to prove that it is transitive we observe that for any measures $\lambda, \mu, \nu$ 
	\begin{multline}
		J(\mu, \nu) = \int_0^1 g(F_\mu^{-1} - F_\nu^{-1})\, \diff x
		\le \int_0^1 \bigl( A + B g(F_\mu^{-1} - F_\lambda^{-1}) \\
		+ Bg(F_\nu^{-1} - F_\lambda^{-1})\bigr)\, \diff x
		= A + BJ(\mu, \lambda) + BJ(\nu, \lambda) < +\infty.\qedhere
	\end{multline}
\end{proof}

Denote the equivalence class of~$\nu$ with respect to the relation~$\sim$ by $C(\nu) := \{\mu\in\mathcal{P}(\rset)\colon \mu \sim \nu\}$.
Observe that for any $\nu \in \mathcal{P}(\rset)$ the function $J(\cdot, \nu)$ is measurable with respect to the Borel $\sigma$-algebra generated by the topology of weak convergence in~$\mathcal{P}(\rset)$, and therefore $C(\nu)$ is measurable for any~$\nu$.
Observe also that for $g(x) = |x|^p$, $p\ge 1$, the corresponding Wasserstein space~$W_p$ coincides with $C(\delta_0)$, the equivalence class of a Dirac unit mass at the origin.

\begin{lemma}
	\emph{(1)} The barycenter $\bary(\mu_i, \lambda_i)_{1\le i\le n}$ is defined iff the measures $\mu_1, \dots,\allowbreak \mu_n$ all lie in the same equivalence class.

	\emph{(2)} If the barycenter $\bary(P_\mu)$ exists, then $\mathop{\mathrm{supp}} P_\mu$ belongs to a single equivalence class, i.e., there exists~$\nu_0 \in \mathcal{P}(\rset)$ such that $P_\mu(C(\nu_0)) = 1$. 
\end{lemma}

\begin{proof}
  (1) If $\mu_1, \dots, \mu_n$ belong to different equivalence classes, then the quantity $\sum_{1\le i\le n} \lambda_i J(\mu_i, \nu)$ is infinite for any $\nu\in \mathcal{P}(\rset)$, and no barycenter exists. Conversely, if all measures are equivalent, then $\sum_{1\le i\le n} \lambda_i J(\mu_i, \nu)$ is finite e.g.\ for $\nu = \mu_1$, and a unique barycenter exists by Theorem~\ref{thm:emp_bar}.

  (2) Let $P_\mu(C(\nu)) < 1$ for all~$\nu\in \mathcal{P}(\rset)$.
  Then $P_\mu(\mathcal{P}(\rset)\setminus C(\nu)) = P_\mu(\bm\mu \notin C(\nu)) = P_\mu(J(\bm\mu, \nu) = +\infty) > 0$ and $\ave J(\bm\mu, \cdot) = +\infty$ on~$\mathcal{P}(\rset)$, which again implies there is no barycenter.
\end{proof}

Under condition~\eqref{eq:growth_rate}, convergence of Fr{\'e}chet barycenters can be characterized in terms of the transportation cost~$J$ defined in~\eqref{eq:8}.

\begin{theorem}
	\label{thm:strong_conv}
	Let $\{\bm\mu_n\}_{n \ge 1}\subset \mathcal{P}(\rset)$ be a sequence of i.i.d. random elements with distribution $P_\mu$.
	Then, if there exists Fr{\'e}chet barycenter $\nu^* := \bary(P_\mu)$, the sequence $\overline{\bm\mu}_n := \bary(\bm\mu_i, 1 / n)_{1\le i\le n}$ of empirical Fr{\'e}chet barycenters satisfies
	\[\lim_{n\to \infty} J(\overline{\bm\mu}_n, \nu^*) = \lim_{n\to \infty} J(\nu^*, \overline{\bm\mu}_n) = 0\quad \text{a.s.}.\]
\end{theorem}

\begin{proof}
	Consider the function
	\[\psi(x) := \arg\min_{y\in \rset} \ave g(F_{\bm\mu}^{-1}(x) - y).\]
	According to Theorem~\ref{thm:prob_bar}, $F_{\nu^*}^{-1}(x) = \psi(x)$ a.e. and
	\[\ave J(\bm\mu, \nu^*) = \int_0^1 \ave g(F_{\bm\mu}^{-1} - \psi)\, \diff x < +\infty.\]
	It follows that 
	\[\lim_{k\to \infty} \int_0^{\frac1k} \ave g(F_{\bm\mu}^{-1} - \psi)\, \diff x = \lim_{k\to \infty} \int_{1 - \frac1k}^1 \ave g(F_{\bm\mu}^{-1} - \psi)\, \diff x = 0.\]
		
	For the random variables $\int_0^{\frac1k} g(F_{\bm\mu_i}^{-1} - \psi)\, \diff x$, $k\ge 1$, the strong law of large numbers implies that
	\[\frac1n \sum_{i=1}^n \int_0^{\frac1k} g(F_{\bm\mu_i}^{-1} - \psi)\, \diff x
	\xrightarrow[n\to\infty]{} \ave \int_0^{\frac1k} g(F_{\bm\mu}^{-1} - \psi)\, \diff x < +\infty\ \text{a.s.}\]
	Using the bound~\eqref{eq:growth_rate} on~$g(\cdot)$ and the construction of the empirical Fr{\'e}chet barycenter (Theorem~\ref{thm:emp_bar}), we get
	\begin{multline}
		\label{eq:lim_1}
		\int_0^{\frac1k} g(F_{\overline{\bm\mu}_n}^{-1} - \psi)\, \diff x
		\le \frac{A}k + B \int_0^{\frac1k} \frac1n \sum_{i=1}^n \bigl(g(F_{\bm\mu_i}^{-1} - F_{\overline{\bm\mu}_n}^{-1}) + g(F_{\bm\mu_i}^{-1} - \psi)\bigr)\, \diff x\le {}\\
		{}\le \frac{A}k + \frac{2 B}n \sum_{i=1}^n \int_0^{\frac1k}
		g(F_{\bm\mu_i}^{-1} - \psi)\, \diff x\\
		\xrightarrow[n\to \infty]{} \frac{A}k + 2 B \ave \int_0^{\frac1k} g(F_{\bm\mu}^{-1} - \psi)\, \diff x
		\xrightarrow[k\to\infty]{} 0.
	\end{multline}
	Similarly one can prove that 
	\begin{equation}
		\label{eq:lim_2}
		\int_{1 - \frac1k}^1 g(F_{\overline{\bm\mu}_n}^{-1} - \psi)\, \diff x\xrightarrow[n\to \infty]{} \frac{A}k + 2 B \ave \int_{1 - \frac1k}^1 g(F_{\bm\mu}^{-1} - \psi)\, \diff x\xrightarrow[k\to\infty]{} 0\ \text{a.s.}
	\end{equation}
	But by Theorem~\ref{thm:prob_bar} $\overline{\bm\mu}_n$ weakly converges to~$\nu^*$ a.s., hence for all $k\ge 1$
	\begin{equation}
		\label{eq:lim_3}
		\int_{\frac1k}^{1 - \frac1k} g(F_{\overline{\bm\mu}_n}^{-1} - \psi)\, \diff x
		\xrightarrow[n\to\infty]{} \int_{\frac1k}^{1 - \frac1k} g(F_{\nu^*}^{-1} - \psi)\, \diff x = 0\ \text{a.s.}
	\end{equation}
	Now it follows from formulas~\eqref{eq:lim_1},~\eqref{eq:lim_2} and~\eqref{eq:lim_3} that 
	\[J(\overline{\bm\mu}_n, \nu^*) = \int_0^1 g(F_{\overline{\bm\mu}_n}^{-1} - \psi)\, \diff x
	\xrightarrow[n\to\infty]{} J(\nu^*, \nu^*) = 0\ \text{a.s.}\]		
	
	In the same way, one can obtain that a.s.
	\[J(\nu^*, \bm\nu_n)\xrightarrow[n\to\infty]{} J(\nu^*, \nu^*) = 0.\qedhere\]
\end{proof}

It remains to clarify the relation between the convergence defined in terms of the transportation cost ($\mu_n \to \mu$ iff $J(\mu_n, \mu) \to 0$), which is used in the preceding theorem, and the weak convergence of measures $\mu_n \rightharpoonup \mu$.

\begin{lemma}
	\label{lem:Jweak}
	Let the sequence $\{\mu_n\}_{n\ge 1}$ be such that $J(\mu_n, \mu^*)\to 0$ for some $\mu^*\in \mathcal{P}(\rset)$; then $\mu_n \rightharpoonup \mu^*$.
\end{lemma}

\begin{proof}
	Suppose on the contrary that the weak convergence of~$\mu_n$ to~$\mu^*$ does not hold.
	Then there exists a point $0 < x_0 < 1$ where $F^{-1}_{\mu^*}$ is continuous but the sequence $F^{-1}_{\mu_n}(x_0)$ does not converge to~$F^{-1}_{\mu^*}(x_0)$.
	Assume specifically that there exists $\epsilon > 0$ such that $F^{-1}_{\mu_n}(x_0) \ge F^{-1}_{\mu^*}(x_0) + 2\epsilon$ for a suitable subsequence, which we still denote $\mu_n$.
	Monotonicity of $F^{-1}_{\mu_n}$ and continuity of~$F^{-1}_{\mu^*}$ at~$x_0$ imply that $F^{-1}_{\mu_n}(x) \ge F^{-1}_{\mu^*}(x) + \epsilon$ for $x_0 < x < x_0 + \delta$ with some $\delta > 0$.
	Then
	\begin{displaymath}
		J(\mu_n, \mu^*) \ge \int_{x_0}^{x_0 + \delta} g(F^{-1}_{\mu_n} - F^{-1}_{\mu^*})\, \diff x
		\ge \int_{x_0}^{x_0 + \delta} g(\epsilon)\, \diff x = g(\epsilon) \delta > 0,
	\end{displaymath}
	which contradicts the assumption $J(\mu_n, \mu^*) \to 0$.
\end{proof}

\begin{theorem}
	\label{thm:strongweak}
	For a sequence of measures $\{\mu_n\}_{n\ge 1}$ and a measure~$\mu^*$ the following conditions are equivalent:
	\begin{enumerate}
		\item $\mu_n \rightharpoonup \mu^*$ and $J(\mu_n, \nu) \to J(\mu^*, \nu)$ for all $\nu \in \mathcal{P}(\rset)$;
		\item $\mu_n \rightharpoonup \mu^*$ and there exists a measure $\nu_0$ such that $J(\mu_n, \nu_0) \to J(\mu^*, \nu_0) < +\infty$;
		\item $J(\mu_n, \mu^*) \to 0$.
	\end{enumerate}
\end{theorem}

\begin{proof}
	Obviously (1) implies~(2). To show that (2) implies~(3), observe that for any $\epsilon > 0$ we have
	\begin{displaymath}
		\int_\epsilon^{1 - \epsilon} g(F^{-1}_{\mu_n} - F^{-1}_\nu)\, \diff x
		\to \int_\epsilon^{1 - \epsilon} g(F^{-1}_{\mu^*} - F^{-1}_\nu)\, \diff x
	\end{displaymath}
	due to the weak convergence of measures.
	Since
	\begin{multline}
		J(\mu_n, \nu) = \Bigl(\int_0^\epsilon + \int_\epsilon^{1 - \epsilon} + \int_{1 - \epsilon}^1\Bigr)\, g(F^{-1}_{\mu_n} - F^{-1}_\nu)\, \diff x\\
		\xrightarrow[n\to\infty]{} J(\mu^*, \nu) = \Bigl(\int_0^\epsilon + \int_\epsilon^{1 - \epsilon} + \int_{1 - \epsilon}^1\Bigr)\, g(F^{-1}_{\mu^*} - F^{-1}_\nu)\, \diff x,
	\end{multline}
	we have
	\begin{displaymath}
		\Bigl(\int_0^\epsilon + \int_{1 - \epsilon}^1\Bigr)\, g(F^{-1}_{\mu_n} - g^{-1}_\nu)\, \diff x 
		\xrightarrow[n\to\infty]{} \Bigl(\int_0^\epsilon + \int_{1 - \epsilon}^1\Bigr)\, g(F^{-1}_{\mu^*} - g^{-1}_\nu)\, \diff x,
	\end{displaymath}
	where the right-hand side vanishes as~$\epsilon \to 0$ because the integral $\int_\epsilon^{1 - \epsilon} g(F^{-1}_{\mu^*} - F^{-1}_\nu)\, \diff x$ converges to~$J(\mu^*, \nu)$.
	Thus
	\begin{multline}
		\Bigl(\int_0^\epsilon + \int_{1 - \epsilon}^1\Bigr)\, g(F^{-1}_{\mu_n} - F^{-1}_{\mu^*})\, \diff x \\
		\le \Bigl(\int_0^\epsilon + \int_{1 - \epsilon}^1\Bigr)\, (A + B g(F^{-1}_{\mu_n} - F^{-1}_\nu)
		+ B g(F^{-1}_{\mu^*} - F^{-1}_\nu))\, \diff x \\
		\xrightarrow[n\to\infty]{} 2A\epsilon + 2B\Bigl(\int_0^\epsilon + \int_{1 - \epsilon}^1\Bigr)\,
		g(F^{-1}_{\mu^*} - F^{-1}_\nu)\, \diff x \xrightarrow[\epsilon \to 0]{} 0.
	\end{multline}
	Using the weak convergence again we observe that
	\begin{displaymath}
		\int_\epsilon^{1 - \epsilon} g(F^{-1}_{\mu_n} - F^{-1}_{\mu^*})\, \diff x \to 0
	\end{displaymath}
	for any $\epsilon > 0$.
	Thus
	\begin{displaymath}
		J(\mu_n, \mu) = \int_0^1 g(F^{-1}_{\mu_n} - F^{-1}_{\mu^*})\, \diff x \to 0.
	\end{displaymath}
	
	It remains to prove that (3) implies~(1).
	Fix a measure $\nu \in \mathcal{P}(\rset)$.
	By Lemma~\ref{lem:Jweak} convergence $J(\mu_n, \mu^*) \to 0$ implies weak convergence $\mu_n \rightharpoonup \mu^*$.
	Consider
	\begin{multline}
		\int_0^\epsilon g(F^{-1}_{\mu_n} - F^{-1}_\nu)\, \diff x
		\le \int_0^\epsilon (A + Bg(F^{-1}_{\mu_n} - F^{-1}_{\mu^*}) + Bg(F^{-1}_\nu - F^{-1}_{\mu^*})]\, \diff x \\
		\le \int_0^\epsilon (A + Bg(F^{-1}_{\mu_n} - F^{-1}_{\mu^*}) + AB + B^2 g(F^{-1}_{\mu^*} - F^{-1}_\nu))\, \diff x \\
		\xrightarrow[n\to\infty]{} A(1 + B)\epsilon + B^2 \int_0^\epsilon g(F^{-1}_{\mu^*} - F^{-1}_\nu)\, \diff x
		\xrightarrow[\epsilon\to 0]{} 0.
	\end{multline}
	A similar result holds for~$\int_{1 - \epsilon}^1 g(F^{-1}_{\mu_n} - F^{-1}_\nu)\, \diff x$.
	Now the weak convergence $\mu_n \rightharpoonup \mu^*$ implies that
	\begin{displaymath}
		\int_\epsilon^{1 - \epsilon} g(F^{-1}_{\mu_n} - F^{-1}_\nu)\, \diff x
		\xrightarrow[n\to\infty]{} \int_\epsilon^{1 - \epsilon} g(F^{-1}_{\mu^*} - F^{-1}_\nu)\, \diff x
		\xrightarrow[\epsilon\to 0]{} J(\mu^*, \nu).
	\end{displaymath}
	Gathering the results for $\int_0^\epsilon$, $\int_\epsilon^{1 - \epsilon}$, and~$\int_{1 - \epsilon}^1$, we obtain $J(\mu_n, \nu) \to J(\mu^*, \nu)$.
\end{proof}

\begin{remark}
	The arguments of $J(\cdot, \cdot)$ can be simultaneously swapped in each of the conditions (1)--(3) without violating the theorem.
	In particular $J(\mu_n, \mu^*) \to 0$ is equivalent to $J(\mu^*, \mu_n) \to 0$. 
\end{remark}

Now we fix some $\nu_0 \in \mathcal{P}(\rset)$. By Theorem~\ref{thm:strongweak}, the convergence $J(\mu_n, \mu^*) \to 0$ implies that $J(\mu_n, \nu) \to J(\mu^*, \nu)$ for all~$\nu \in C(\nu_0)$.
This enables us to define on~$C(\nu_0)$ the following topology, which is at least as strong as the topology of weak convergence.

\begin{definition}
  The balls $B_r(\nu) := \{\mu\in C(\nu_0)\colon J(\mu, \nu) < r\}$ form a basis of a \emph{topology $\tau_J$} on~$C(\nu_0)$.
\end{definition}

In particular, Theorem~\ref{thm:strong_conv} states that empirical Fr{\'e}chet barycenters $\bary(\bm\mu_i, 1 / n)$ convergence to $\bary(P_\mu)$ in topology $\tau_J$ a.s.

\appendix

\section*{Appendix: Proof of the measurability lemma}
\label{sec:appendix}

\begin{proof}[Proof of Lemma~\ref{lm:1}]
	We have to show that for any $a \in\rset$ the set $U_a := K^{-1}((a, +\infty))$ is $\mathcal{B}(\mathcal{P}(\rset)) \otimes \mathcal{B}((0, 1))$-measurable.
	Fix some $x\in\rset$ and consider the function
	\begin{equation}
		\label{eq:19}
		k_x(\mu) := K(\mu, x)= F^{-1}_\mu(x)
	\end{equation}
	on~$\mathcal{P}(\rset)$; we will show first that $k^{-1}_x((a, +\infty)) \subset \mathcal{P}(\rset)$ is open.

	Take an arbitrary measure $\bar\mu \in U_a$ and denote $y := k_x(\bar\mu) = F^{-1}_{\bar\mu}(x) > a$.
	It suffices to show that there exists an open neighbourhood~$V$ of~$\bar\mu$ in the weak topology of~$\mathcal{P}(\rset)$ such that $V \subset k^{-1}_x((a, +\infty))$.

	The left continuity of~$F_{\bar\mu}$ ensures that $F_{\bar\mu}(\frac{a + y}2) < x - \delta$ for some $\delta > 0$.
	Take a continuous function $v\colon \rset \to [0, 1]$ such that $v(t) = 1$ for $t \le a$ and $v(t) = 0$ for $t \ge (a + y)/2$.
	It follows that
	\begin{equation}
		\label{eq:20}
		\int_\rset v\,\diff\bar\mu \le \int_{-\infty}^{\frac{a + y}2}\, \diff\bar\mu
		= F_{\bar\mu}\Bigl(\frac{a + y}2\Bigr) < x - \delta.
	\end{equation}
	On the other hand, for any $\mu \in \mathcal{P}(\rset)$ such that $F_\mu(a + 0) \ge x$ we have
	\begin{equation}
		\label{eq:21}
		\int_\rset v\,\diff\mu \ge \int_{-\infty}^a\, \diff\mu = F_\mu(a + 0) \ge x,
	\end{equation}
	which implies $\int_\rset v\,\diff\mu - \int_\rset v\,\diff\bar\mu > \delta$.
	Therefore for all measures~$\nu$ in the weak neighborhood
	\begin{equation}
		\label{eq:22}
		V_\delta(\bar\mu) = \Bigl\{\nu\in \mathcal{P}(\rset)\colon
		\Bigl|\int_\rset v\,\diff\nu - \int_\rset v\, \diff\bar\mu\Bigr| < \delta\Bigr\}.
	\end{equation}
	it follows that $F_\nu(a + 0) < x$, or $k_x(\nu) > a$.
	Thus the set $k^{-1}_x((a, +\infty))$ is open in the weak topology of~$\mathcal{P}(\rset)$.

	From monotonicity of the inverse cumulative distribution function we obtain that $A \times [x, 1) \subset U_a$ whenever $A\times \{x\} \in U_a$.
	The left continuity of~$F_\mu$ implies that for any $(\mu, x) \in U_a$ there exists a rational $s \in (0, x]$ such that $F^{-1}_\mu(s) > a$, whence $(\mu, s) \in U_a$.
	Therefore
	\begin{equation}
		\label{eq:23}
		U_a = \cup_{x\in (0, 1)} k^{-1}_x((a, +\infty)) \times \{x\}
		= \cup_{s\in (0, 1) \cap \mathbb{Q}} k^{-1}_s((a, +\infty)) \times [s, 1).
	\end{equation}
	Thus $U_a$ is measurable because it is a countable union of measurable sets, and $K(\cdot, \cdot)$ is measurable.
\end{proof}

\end{document}